\documentclass[12pt]{article}

\usepackage{verbatim}
\usepackage{enumerate}

\usepackage{color, float}

\usepackage{amssymb,amsthm,amsmath}
\usepackage{latexsym}
\usepackage{graphicx}
\usepackage{url}
\usepackage{fullpage}

\newtheorem{theorem}{Theorem}
\newtheorem{lemma}[theorem]{Lemma}
\newtheorem{claim}[theorem]{Claim}

\newtheorem{cor}[theorem]{Corollary}

\newtheorem*{restate2}{Theorem 2}
\newtheorem*{restate3}{Theorem 3}

\newcommand{\h}{\mathcal{H}}

\newcommand\ex{\ensuremath{\mathrm{ex}}}

\title{Tur\'an numbers for hypergraph star forests}

\author{
Omid Khormali
\qquad
Cory Palmer\\
\small Department of Mathematical Sciences\\[-0.8ex]
\small University of Montana\\[-0.8ex]
\small Missoula, Montana 59812, USA.\\
\small \texttt{omid.khormali@umontana.edu, cory.palmer@umontana.edu}
}

\begin{document}

\maketitle

\begin{abstract}
Fix a graph $F$. We say that a graph is {\it $F$-free} if it does not contain $F$ as a subgraph. The {\it Tur\'an number} of $F$, denoted $\ex(n,F)$, is the maximum number of edges possible in an $n$-vertex $F$-free graph. The study of Tur\'an numbers is a central problem in graph theory. The goal of this paper is to generalize a theorem of Lidick\'y, Liu and Palmer [{\it Electron.\ J.\ of Combin.}\ {\bf 20} (2016)] that determines $\ex(n,F)$ for $F$ a forest of stars. In particular, we consider  generalizations of the problem to three different well-studied hypergraph settings and in each case we prove an asymptotic result for all reasonable parameters defining our ``star forests''.
\end{abstract}

\section{Introduction}

Let $\mathcal{F}$ be an $r$-uniform hypergraph. A hypergraph is {\it $\mathcal{F}$-free} if it has no subgraph isomorphic to $\mathcal{F}$. The {\it Tur\'an number} of $\mathcal{F}$ is the maximum number of hyperedges in an $r$-uniform $n$-vertex $\mathcal{F}$-free hypergraph. We denote this maximum by $\ex_r(n,\mathcal{F})$. When $r=2$ we are considering the problem to determine $\ex(n,F) = \ex_2(n,F)$ where $F$ is an ordinary graph. This is a central and well-studied problem in graph theory. For example, Tur\'an's seminal theorem \cite{turan} determines exactly the Tur\'an number for a complete graph $K_k$.
The fundamental Erd\H{o}s-Stone-Simonovits theorem \cite{Erdos-Simonovits, Erdos-Stone} gives the following asymptotic result for all $k$-chromatic graphs $F$:
\[
\ex(n,F) = \left(1-\frac{1}{k-1} \right) \frac{n^2}{2} + o(n^2).
\]
Note that when $F$ is bipartite this only gives $\ex(n,F) = o(n^2)$. Determining Tur\'an numbers for bipartite graphs remains an active area of research. See the survey of F\"uredi and Simonovits \cite{Furedi-Simonovits} for an extensive history. Classic bounds are given by Erd\H os and Gallai \cite{EG} for paths, K\H ov\'ari, S\'os and Tur\'an \cite{kst} for complete bipartite graphs and Bondy and Simonovits \cite{Bondy} for even cycles. Erd\H os and S\'os \cite{Er-sos} conjectured $\ex(n,F) \leq \frac{t-2}{2}n$ when $F$ is a tree on $t$ vertices. Ajtai, Koml\'os, Simonovits and Szemer\'edi announced a proof of this conjecture. The conjecture is also known to hold for various classes of specific trees. Unlike what is conjectured for trees, the Tur\'an number of forests depends heavily on the structure of the forest. 

Let $S_\ell$ be the star with $\ell$ edges (i.e., $\ell$ distinct edges all sharing the same common vertex) and let $P_\ell$ denote\footnote{Note that in related works the notation $P_\ell$ is sometimes used for the $\ell$-vertex path and at other times for the $\ell$-edge path.} the path of length $\ell$, i.e., the path with $\ell$ edges and $\ell+1$ vertices (it will be convenient later when dealing with hypergraphs that we keep track of the number of hyperedges rather than vertices).
For a graph (or hypergraph) $F$ let $k \cdot F$ denote the graph composed of $k$ pairwise vertex-disjoint copies of the graph $F$. Bushaw and Kettle \cite{BK} gave the following bounds for a forest of paths where $k\geq 2$, $\ell \geq 3$ and $n$ large enough,
\[
\ex(n, k\cdot P_{\ell})=\left(k\left\lfloor\frac{\ell+1}{2}\right\rfloor-1\right) \left(n-k\left\lfloor\frac{\ell+1}{2}\right\rfloor+1\right)+\binom{k\left\lfloor\frac{\ell+1}{2}\right\rfloor-1}{2}+c_{\ell}
\] 
where $c_{\ell}= 1$ if $\ell+1$ is odd, and $c_{\ell}= 0$ if $\ell+1$ is even.
Compare this bound with a special case of a theorem of Lidick\'y, Liu and Palmer \cite{LLP} for a forest of stars (see also \cite{JPS} for a simplified proof).

\begin{theorem}[Lidick\'y, Liu and Palmer,
\cite{LLP}]\label{LLP-star}
Fix integers $\ell, k \geq 1$. Then for $n$ be large enough,
\[ 
\ex(n,k \cdot S_{\ell}) = \left\lfloor \frac{\ell-1}{2}(n-k+1)\right\rfloor + (k-1)(n-k+1)+\binom{k-1}{2}.
\]
\end{theorem}

Determining Tur\'an numbers for hypergraphs of uniformity $r\geq 3$ has been significantly more difficult than in the graph case. For example, we do not know the Tur\'an number of $K_4^3$, the $3$-uniform $4$-vertex complete hypergraph. Another open example is the star $S_\ell^+$ which is the $r$-uniform hypergraph consisting of $\ell$ hyperedges all sharing exactly one common vertex. This problem is related to the sunflower problem \cite{ER} which suggests why it appears to be difficult.

Motivated by Theorem~\ref{LLP-star} and the difficulty of determining hypergraph Tur\'an numbers, the goal of this paper is to find Tur\'an numbers of various analogues of ``star forests'' in the hypergraph setting. Before stating our main theorems we need several definitions.

Given a graph $F$, the {\it expansion} of $F$ is the $r$-uniform hypergraph $F^+$ constructed by adding $r-2$ new distinct vertices to each edge of $F$. Note that when $r=2$, then the expansion $F^+$ is simply the graph $F$.                      Tur\'an numbers for various expansions have been investigated.
Mubayi \cite{mubayi} and Pikhurko \cite{pikhurko} considered the case when $F$ is a complete graph. F\"uredi and Jiang \cite{FJ}; F\"uredi, Jiang, and  Seiver \cite{FJS}; and F\"uredi \cite{F-tree} examined the case when  $F$ is a path, cycle, or tree, respectively.
A series of papers 
\cite{kmv I, kmv II, kmv III} by Kostochka, Mubayi and  Verstra\"{e}te also consider 
expansions for paths, cycles, trees, as well as other graphs. 
Bushaw and Kettle \cite{BK2} consider the case $k \cdot P_\ell^+$, i.e, a forest $k$ disjoint expansions of the path $P_\ell$. The case when $F$ is the expansion of a star $S_\ell$ has considerable history. See the survey of Mubayi and Verstra\"ete \cite{mv} for an overview.
Our first two main theorems involve a hypergraph forest composed of the expansion of stars, i.e., $k \cdot S_\ell^+$. First we prove:

\begin{theorem}\label{forest-nonlinear}
Fix integers $\ell,k\geq 1$ and $r\geq 2$. Then for $n$ large enough,
\[
\ex_r(n,k \cdot S_\ell^+)= \binom{n}{r}-\binom{n-k+1}{r}+ \ex_r(n-k+1, S_\ell^+).
\]
\end{theorem}

A theorem of Duke and Erd\H os \cite{DE}
gives $\ex_r(n,S_\ell^+) = \Theta(n^{r-2})$, so Theorem~\ref{forest-nonlinear} gives the asymptotic bound
\[
\ex_r(n,k \cdot S_\ell^+) \sim  \frac{k-1}{(r-1)!} n^{r-1}.
\]

Note that an expansion is a {\it linear} hypergraph, i.e., a hypergraph such that every pair of hyperedges share at most one vertex. It is natural to consider extremal problems where the host hypergraph is linear when the forbidden hypergraph is linear. To that end let $\ex_r^{\operatorname{lin}}(n,F)$ be the maximum number of hyperedges in an $r$-uniform $n$-vertex linear hypergraph containing no subhypergraph isomorphic to $F$. An interesting example is $\ex_3^{\operatorname{lin}}(n,C_3^+)$ which is equivalent to the famous $(6,3)$-problem.
Tur\'an numbers in linear host hypergraphs have been examined in \cite{CGJ, EGM, gmv, lv, craig}. However, most of these papers do not deal specifically with the case when an expansion $F^+$ is forbidden.
In this case we prove:


\begin{theorem}\label{forest-linear}
Fix integers $\ell,k \geq 1$ and $r\geq 2$. Then
for $n$ large enough,
\[
\ex_r^{\operatorname{lin}}(n,k\cdot S_\ell^+) \leq \left(\frac{\ell-1}{r}+\frac{k-1}{r-1}\right)(n-k+1) +\frac{\binom{k-1}{2}}{\binom{r}{2}}.
\]
Furthermore,
this bound is sharp asymptotically.
\end{theorem}

For our third main theorem we need a further definition due to Gerbner and Palmer \cite{Gerbner_Palmer}.
For a graph $F$, we say that a hypergraph $\h$ is a {\it Berge}-$F$ if there is an injection $f: V(F) \rightarrow V(\h)$ and bijection $f':E(F)\rightarrow E(\h)$ such that for every edge $uv\in E(F)$ we have $\{f(u),f(v)\}\subseteq f'(uv)$. Alternatively, $\h$ is Berge-$F$ if we can embed a distinct graph edge into each hyperedge of $\h$ to obtain a copy of $F$. Note that for a fixed $F$ there are many different hypergraphs that are a Berge-$F$ and a fixed hypergraph $\h$ can be a Berge-$F$ for more than one graph $F$.

We use the term {\it Berge-$F$-free} for hypergraphs that have no subhypergraph isomorphic to any Berge-$F$.
For a fixed graph $F$, let $\ex_r(n,\text{Berge-}F)$ denote the maximum number of hyperedges in an $r$-uniform $n$-vertex Berge-$F$-free hypergraph.

The behavior of $\ex_r(\text{Berge-}F)$ has been investigated in a number of recent manuscripts. For example, see \cite{FuLa, Gyori_triangle, Gyori_Lemons,GyLe4,GyLe} for cycles, \cite{DGMT, GyKaLe} for paths and \cite{Gerbner_Palmer, gmv,Palmer} for complete bipartite graphs. General results are given in \cite{anssal, GMP, G_P2, GMT}.
For a short survey of extremal results for Berge hypergraphs see Subsection 5.2.2 in \cite{gp}.

As the expansion $F^+$ of $F$ is just a particular instance of a Berge-$F$ we have the following trivial inequality
\[
\ex_r(n, \textrm{Berge-}F) \leq \ex_r(n,F^+).
\]
However, in general these two extremal numbers are not asymptotic. Our third main theorem and Theorem~\ref{forest-nonlinear} give an example of when they differ significantly.


\begin{theorem} \label{bergekS}
    Fix integers $\ell,k \geq 1$, $r\geq 3$ and let $n$ be large enough.
    If $r \geq \ell+k-1$, then
	\[
	\ex_r(n,\textup{Berge-}k \cdot S_{\ell}) \leq \frac{\ell-1}{r-k+1}(n-k+1).
	\]
	If $r \leq \ell+k-2$ then,
	\[
	\ex_r(n,\textup{Berge-}k \cdot S_{\ell}) \leq \left(\binom{\ell+k-1}{r}-\binom{k -1}{r}\right)\left \lceil \frac{n-k+1}{\ell}\right \rceil+\binom{k-1}{r}.
	\]
	Furthermore, both upper bounds are sharp asymptotically.
\end{theorem}

Theorem~\ref{bergekS} will be the consequence of two theorems proved in Section~\ref{berge-section}. 
We prove Theorem~\ref{forest-nonlinear} in Section~\ref{linear-section-1} and Theorem~\ref{forest-linear} in Section~\ref{linear-section-2}.
Note that the proofs of Theorem~\ref{forest-nonlinear} and \ref{forest-linear} work when $r=2$ which give two new proofs of Theorem~\ref{LLP-star}.

\smallskip

{\bf Notation.} Notation is generally standard and follows the monograph of Bollob\'as \cite{bollobas}. For a hypergraph $\mathcal{H}$, let $E(\mathcal{H})$ and $V(\mathcal{H})$ denote the hyperedge set and vertex set, respectively. For a pair of vertex sets $A$ and $B$, let $E(A,B)$ denote the set of hyperedges that have at least one vertex in both of $A$ and $B$. We will use the term {\it center} to refer to a vertex of degree $\ell$ in a star with $\ell$ edges. In a star $S_\ell$ or $S_\ell^+$ the center is unique. However, in a Berge-$S_\ell$ there may be more than one vertex eligible to be the center.

\section{Forest of expansions of stars}\label{linear-section-1}

We begin with two classic theorems which are generalized by Theorem~\ref{forest-nonlinear}.

\begin{theorem}[Duke and Erd\H os, \cite{DE}]\label{D-E}
Fix integers $\ell\geq 2$ and $r\geq 3$. Then there exists a constant $c(r)$ such that for $n$ large enough,
\[
\ex_r(n,S_\ell^+) \leq c(r)\ell(\ell-1)n^{r-2}.
\]
\end{theorem}

An easy lower bound of order $n^{r-2}$ on $\ex_r(n,S_\ell^+)$ comes from an $r$-uniform hypergraph of size $\binom{n-2}{r-2}$ consisting of all $r$-sets containing a fixed pair of vertices.

Let $M_k^+$ be a set of $k$ pairwise-disjoint hyperedges of size $r$, i.e., a matching of $k$ hyperedges.

\begin{theorem}[Erd\H os, \cite{Erdos2}] \label{E2}
Fix integers $k \geq 1$ and $r\geq 2$. Then for $n$ large enough,
\[
\operatorname{ex}_r(n,M_k^+)=\binom{n}{r}-\binom{n-k+1}{r}.
\]
\end{theorem}

The problem to improve the threshold on $n$ for which Theorem~\ref{E2} holds has attracted considerable attention and is known as the Erd\H os Matching Conjecture.
See \cite{FK} for the best-known bound and further historical details.

As we may view $M_k^+$ as a forest of stars each of size $1$, i.e., $k \cdot S^+_1$, Theorem~\ref{E2} serves as an initial case for the Tur\'an number of a forest of the expansion of stars. 
We are now ready to prove Theorem~\ref{forest-nonlinear} which we restate here for convenience. Our proof is an adaptation of the proof of Theorem~\ref{E2}. 

\begin{restate2}
Fix integers $\ell,k\geq 1$ and $r\geq 2$. Then for $n$ large enough,
\[
\ex_r(n,k \cdot S_\ell^+)= \binom{n}{r}-\binom{n-k+1}{r}+ \ex_r(n-k+1, S_\ell^+).
\]
\end{restate2}

\begin{proof}
For the lower bound, consider an $r$-uniform $n$-vertex hypergraph constructed as follows. Let $A$ and $B$ be sets of $k-1$ and $n-k+1$ vertices, respectively. First we embed an $S_\ell^+$-free hypergraph with $\ex_r(n-k+1,S_\ell^+)$ hyperedges into $B$. Next we add every $r$-set that is incident to $A$ to our hypergraph. It is easy to see that this hypergraph has exactly as many hyperedges as in the statement of the theorem. Moreover, as $B$ contains no copy of $S_\ell^+$ each such subgraph must contain at  least one vertex of $A$. Therefore, there is no $k \cdot S_\ell^+$ subgraph.

We now continue with the upper bound.
Let $\mathcal{H}$ be an $r$-uniform $n$-vertex hypergraph with
\[
|E(\mathcal{H})| > \binom{n}{r} - \binom{n-k+1}{r} + \ex_r(n-k+1,S_\ell^+).
\]
We will show that $\mathcal{H}$ contains a copy of $k \cdot S_\ell^+$.
We proceed by induction on $k$. For $k=1$ the base case is immediate as 
$|E(\mathcal{H})| > \ex_r(n,S_\ell^+)$. So let $k>1$ and assume the statement holds for $k-1$. We distinguish two cases based on the maximum degree $\Delta(\h)$ of $\h$.

\bigskip

{\bf Case 1:} The maximum degree satisfies
\[
\Delta(\h) < \frac{1}{(k-1)((r-1)\ell+1)} \left( \binom{n}{r} -\binom{n-k+1}{r}\right).
\]

\bigskip

Consider a copy of $t \cdot S_\ell^+$ in $\mathcal{H}$ such that $t$ is maximal. We claim that $t \geq k$. Indeed, if $t < k$, then at most $(k-1)((r-1)\ell+1)$ vertices are spanned by the $t \cdot S_\ell^+$. Removing these vertices (and the incident hyperedges) leaves at least
\begin{align*}
|E(\mathcal{H})| - (k-1)((r-1)\ell+1) \cdot \Delta(\mathcal{H})  & > 
|E(\mathcal{H})| - \binom{n}{r} - \binom{n-k+1}{r} \\ & > \ex_r(n-k+1,S_\ell^+)
\end{align*}
hyperedges. Therefore, there is a copy of $S_\ell^+$ that is vertex-disjoint from the $t \cdot S_\ell^+$. This violates the maximality of $t$, a contradiction.

\bigskip
{\bf Case 2:} The maximum degree satisfies
\[
\Delta(\h) \geq \frac{1}{(k-1)((r-1)\ell+1)} \left( \binom{n}{r} -\binom{n-k+1}{r}\right).
\]

\bigskip

Let $x$ be a vertex of maximum degree. Observe that $d(x) \leq \binom{n-1}{r-1}$, so
\begin{align*}
|E(\h)|-d(x) & > \binom{n}{r} - \binom{n-1}{r-1} - \binom{n-k+1}{r} + \ex_r(n-k+1,S_\ell^+)\\
& = \binom{n-1}{r} - \binom{n-k+1}{r} - \ex_r(n-k+1,S_\ell^+).
\end{align*}
Therefore, if we remove $x$ from $\h$ and apply induction to the resulting hypergraph we have a copy of $(k-1)\cdot S_\ell^+$. Now it remains to show that there is a copy of $S_\ell^+$ with center $x$ that is vertex-disjoint from the $(k-1)\cdot S_\ell^+$.

First observe that $x$ and any vertex $y \in V((k-1)\cdot S_\ell^+)$ are contained in at most $\binom{n-2}{r-2}$ common hyperedges. Therefore, the number of hyperedges containing $x$ and a vertex of the $(k-1)\cdot S_\ell^+$ is at most
\[
(k-1)((r-1)\ell+1) \binom{n-2}{r-2} = O(n^{r-2}).
\]
On the other hand, $d(x) = \Omega(n^{r-1})$. Therefore, if we remove the hyperedges of the $(k-1)\cdot S_\ell^+$, we are still left with $\Omega(n^{r-1})$ hyperedges incident to $x$. Applying Theorem~\ref{D-E} to these hyperedges gives a copy of $S_\ell^+$ that is vertex-disjoint from the $(k-1) \cdot S_\ell^+$, i.e, $\h$ contains a copy of $k \cdot S_\ell^+$.
\end{proof}

\section{Forest of expansions of stars in linear hypergraphs}\label{linear-section-2}

We need a simple generalization of a lemma proved in \cite{JPS}.

\begin{lemma}[Average Degree Lemma]\label{ADL}
	Fix positive integers $d$ and $\Delta$ and a constant $0 \leq \epsilon <1$. If $\mathcal{G}$ is a hypergraph with average degree at least $d-\epsilon$ and maximum degree at most $\Delta$, then the number of vertices in $\mathcal{G}$ of degree less than $d$ is at most 
	\[
	\frac{\Delta - d+\epsilon}{\Delta - d+1}n.
	\]
	In particular, the number of vertices in $\mathcal{G}$  of degree at least $d$ is $\Omega(n)$.
\end{lemma}

\begin{proof}
	The sum of the degrees in $\mathcal{G}$ is at least $(d-\epsilon)n$. On the other hand, if $s$ is the number of vertices of degree less than $d$ in $\mathcal{G}$, then the sum of the degrees in $\mathcal{G}$ is at most $(d-1)s+\Delta(n-s)$. Combining these two estimates and solving for $s$ gives the result.
\end{proof}

We are now ready to prove Theorem~\ref{forest-linear} which we restate here for convenience. 

\begin{restate3}
Fix integers $\ell,k \geq 1$ and $r\geq 2$. Then
for $n$ large enough,
\[
\ex_r^{\operatorname{lin}}(n,k\cdot S_\ell^+) \leq \left(\frac{\ell-1}{r}+\frac{k-1}{r-1}\right)(n-k+1) +\frac{\binom{k-1}{2}}{\binom{r}{2}}.
\]

Furthermore,
this bound is sharp asymptotically.
\end{restate3}

\begin{proof}
Let $\mathcal{H}$ be an $r$-uniform $n$-vertex linear hypergraph with no $k\cdot S_\ell^+$ subhypergraph. 
Let $A$ be the vertices in $\mathcal{H}$ of degree at least some fixed (large enough) constant $D=D(\ell,k,r)$. If $|A| \geq k$, then we can greedily embed $k$ pairwise vertex-disjoint copies of $S_\ell^+$ into $\mathcal{H}$. Thus, $|A| \leq k-1$.

Let $\mathcal{H}'$ be the $r$-uniform hypergraph resulting from the removal of the vertices of $A$ (and the hyperedges incident to them) from $\mathcal{H}$. The maximum degree in $\mathcal{H}'$ is less than $D$. If the average degree in $\mathcal{H}'$ is at least $\ell-\epsilon$ for any $\epsilon <1$, then by Lemma~\ref{ADL} we have $\Omega(n)$ vertices of degree at least $\ell$ in $\mathcal{H}'$. In this case we can greedily embed $k$ pairwise vertex-disjoint copies of $S_\ell^+$ into $\mathcal{H}'$, a contradiction. Therefore, the average degree in $\mathcal{H}'$ is at most $\ell-1$. Thus,
\[
|E(\mathcal{H'})| \leq \frac{\ell-1}{r}(n - |A|).
\]
 Let $B$ be the vertices $V(\mathcal{H}')=V(\mathcal{H}) -A$. 
 Now let us count the hyperedges of $\mathcal{H}$ that contain at least one vertex of $A$ and one vertex of $B$. Denote this collection of hyperedges by $E(A,B)$.
 To this end let us count the number of pairs $(h,\{x,y\})$ where $h$ is a hyperedge of $\mathcal{H}$ and $x$ is a vertex in $A \cap h$ and $y$ is a vertex in $B \cap h$. Fixing a hyperedge $h \in E(A,B)$ we have $|A \cap h|$ choices for $x$ and $|B\cap h|$ choices for $y$. Thus the number of pairs $(h,\{x,y\})$ is
 \[
 \sum_{h \in E(A,B)} |A \cap h||B \cap h| \geq \sum_{h \in E(A,B)} (r-1) = |E(A,B)|(r-1).
 \]
On the other hand, for a fixed $x$ and $y$ there is at most one hyperedge containing them as $\mathcal{H}$ is linear. Thus, the number of pairs $(h,\{x,y\})$ is at most
$|A|(n-|A|)$.
Combining these two estimates and solving for $|E(A,B)|$ gives
\[
|E(A,B)| \leq \frac{|A|}{r-1}(n-|A|).
\]
Finally, the maximum number of hyperedges contained completely in $A$ is at most $\binom{k-1}{2}/\binom{r}{2}$ as each pair of vertices in $A$ is contained in at most one hyperedge.
Therefore, the number of hyperedges in $\mathcal{H}$ is
\[
|E(\mathcal{H})| \leq \frac{\ell-1}{r}(n - |A|)+ \frac{|A|}{r-1}(n-|A|) + \frac{\binom{k-1}{2}}{\binom{r}{2}}.
\]
As $|A| \leq k-1$, we have that for $n$ large enough,
\[
|E(\mathcal{H})| \leq \left(\frac{\ell-1}{r}+\frac{k-1}{r-1}\right)(n-k+1) + \frac{\binom{k-1}{2}}{\binom{r}{2}}.
\]

Now let us give a construction that satisfies the sharpness assertion. 
Let $[r]^d$ denote the integer lattice formed by $d$-tuples from $\{1,2,\dots, r\}$. We can think of $[r]^d$ as a hypergraph in the following way: the collection of $d$-tuples that are fixed in all but one coordinate form a hyperedge. 
Thus $[r]^d$ is $r$-uniform and has $r^d$ vertices and $d \cdot r^{d-1}$ hyperedges. Observe that $[r]^d$ is linear as two hyperedges are either disjoint or intersect in exactly one vertex. Furthermore, every vertex is included in exactly $d$ hyperedges, so $[r]^d$ is $d$-regular. Finally, note that the hyperedges of $[r]^d$ can be partitioned into $d$ classes each of which forms a matching. This gives a natural proper hyperedge-coloring of $\mathcal{H}$. We call such a proper hyperedge-coloring a {\it canonical coloring}. See Figure~\ref{cube} for examples of $[r]^d$.

\begin{figure}[H]
\begin{center}
\includegraphics[scale=1.4]{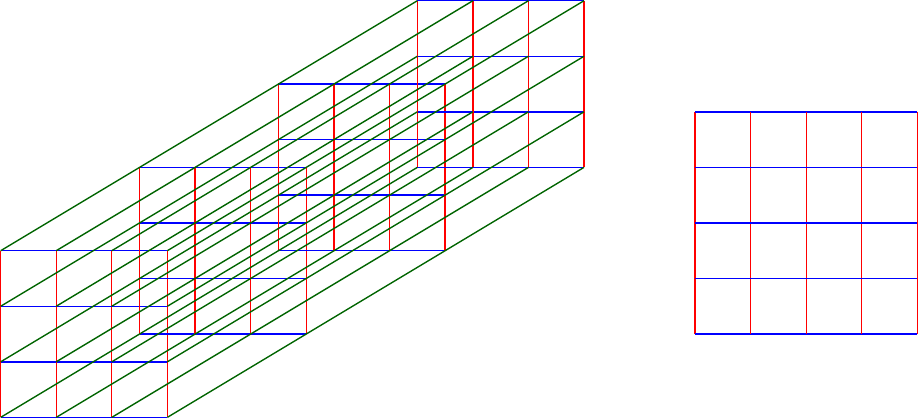}
\end{center}
\caption{The hypergraphs $[4]^3$ and $[5]^2$. Each hyperedge is represented by a line segment. The classes of parallel line segments form the color classes of a canonical coloring}\label{cube}
\end{figure}

The Cartesian product\footnote{Here we use the symbol $\times$ instead of $\square$ to denote the Cartesian product of hypergraphs.} of hypergraphs $\mathcal{H}$ and $\mathcal{G}$ 
is the hypergraph $\mathcal{H}\times \mathcal{G}$ on vertex set $V(\mathcal{H}) \times V(\mathcal{G})$ with hyperedge set
\[
E(\mathcal{H}\times \mathcal{G}) = \{ \{u\} \times e \mid u \in V(\mathcal{H}) \text{ and } e \in E(\mathcal{G})\} \cup \{ \{v\} \times f \mid v \in V(\mathcal{G}) \text{ and } f \in E(\mathcal{H})\}.
\]

Observe that if $\mathcal{H}$ is $r$-uniform and $\mathcal{G}$ is $s$-uniform, then the hyperedges in $\{ \{v\} \times f \mid v \in V(\mathcal{G}) \text{ and } f \in E(\mathcal{H})\}$ are of size $r$ and the hyperedges in $\{ \{u\} \times e \mid u \in V(\mathcal{H}) \text{ and } e \in E(\mathcal{G})\}$ are of size $s$. It is easy to see that if $\mathcal{H}$ and $\mathcal{G}$ are both linear, then $\mathcal{H} \times \mathcal{G}$ is linear. Indeed, two vertices $u_1 \times u_2$ and $v_1 \times v_2$ can be contained in a hyperedge if either $u_1 = u_2$ or $v_1 = v_2$. Observe that $u \times v_1$ and $u \times v_2$ are contained in a hyperedge if and only if $v_1$ and $v_2$ are contained in a hyperedge of $\mathcal{G}$.
 
Now suppose that $n-k+1$ is divisible by $(r-1)^{k-1}$ and $r^{\ell-1}$. Let us construct a hypergraph $\mathcal{H}^*$ as follows. The vertex set of $\mathcal{H}^*$ is partitioned into a set $A^*=\{a_1,a_2,\dots, a_{k-1}\}$ of $k-1$ vertices and  a set $B^*$ of $n-k+1$ vertices partitioned into distinct copies of $[r-1]^{k-1} \times [r]^{\ell-1}$. Let us assume that all hyperedges of size $r-1$ in the copies of $[r-1]^{k-1} \times [r]^{\ell-1}$ inherit their canonical hyperedge-coloring from the original hypergraph $[r-1]^{k-1}$.

The hyperedges of $\mathcal{H}^*$ consist of two types. The first type consist of all hyperedges of size $r$ in the copies of $[r]^{\ell-1}$ in $B^*$. The second type consists of every $h \cup \{a_i\}$ where $a_i \in A^*$ and $h$ is a hyperedge of color $i$ in a copy of $[r-1]^{k-1}$ with a canonical hyperedge-coloring. See Figure~\ref{h-star} for an illustration of this construction. Finally, depending on the size of $k-1$ compared to $r$ we may embed a negligible number of hyperedges into $A$ without violating the linear property of $\mathcal{H}^*$. However, we ignore these potential hyperedges in our construction\footnote{In fact, we can embed all $\binom{k-1}{2}/\binom{r}{2}$ hyperedges of a $2{\text -}(k-1,r,1)$ design into $A$ if the appropriate divisibility conditions are satisfied.}.

Let us confirm that $\mathcal{H}^*$ is linear. Two hyperedges of the first type are hyperedges from copies of the linear hypergraph $[r]^{\ell-1}$ and therefore intersect in at most one vertex. Two hyperedges of the second type are either the same color and therefore intersect in a vertex in $A^*$ but not in $B^*$ or are of different colors and may intersect in $B^*$ in one vertex, but do not intersect in $A^*$. Finally, a hyperedge of the first type and a hyperedge of the second type intersect in at most one vertex by the construction of $[r-1]^{k-1} \times [r]^{\ell-1}$.

\begin{figure}[H]
\begin{center}
{\includegraphics{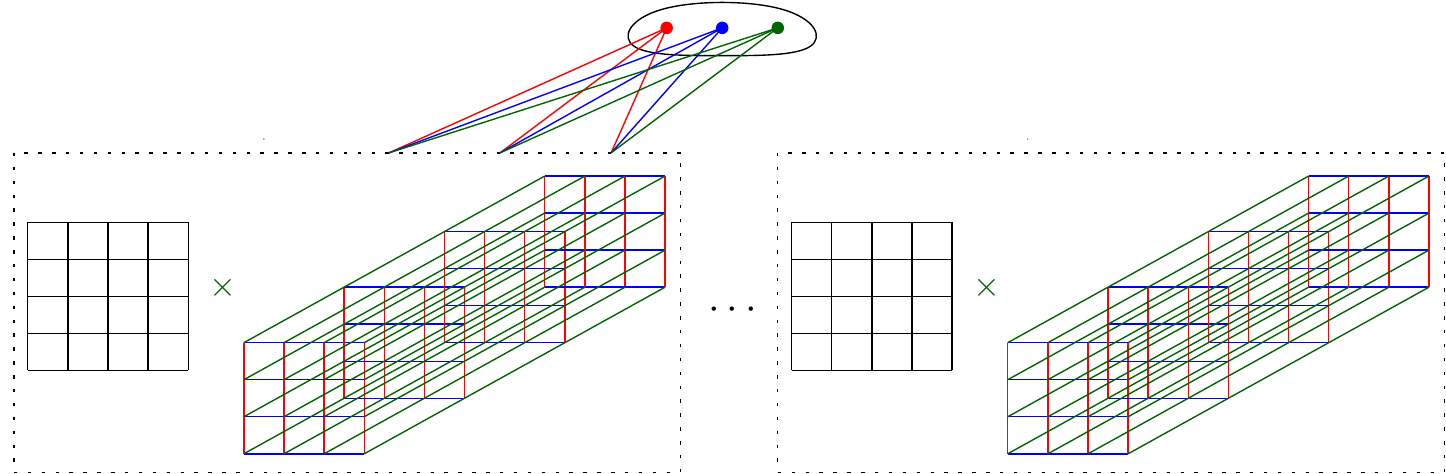}}
\end{center}
\caption{A construction of $\mathcal{H}^*$ for $k-1=3$ and $r=5$}\label{h-star}
\end{figure}

The number of hyperedges in $\mathcal{H}^*$ is 
\begin{align*}
|\mathcal{H}^*| =   |E(A^*,B^*)| + |E(B^*)| & \geq  (k-1) (r-1)^{k-2} \cdot \frac{n-k+1}{(r-1)^{k-1}} + (\ell-1) r^{\ell-2} \cdot\frac{n-k+1}{r^{\ell-1}}\\
& = \left(\frac{\ell-1}{r}+\frac{k-1}{r-1}\right)(n-k+1).
\end{align*}

It now remains to show that $\mathcal{H}^*$ contains no $k \cdot S_\ell^+$. Every vertex in $B^*$ is incident to exactly $\ell-1$ hyperedges that are contained completely in $B^*$. Therefore, any copy of $S_\ell^+$ in $\mathcal{H}^*$ must use at least one vertex from $A^*$. Therefore, there are at most $k-1$ pairwise vertex-disjoint copies of $S_\ell^+$.
\end{proof}

Theorem \ref{forest-linear} gives the following corollary for a matching in the linear setting.

\begin{cor}
Fix integers $k \geq 2$ and $r \geq 2$. Then for $n$ large enough,
\[
\operatorname{ex}_r^{\operatorname{lin}}(n,M_k^+)=\operatorname{ex}_r^{\operatorname{lin}}(n,k \cdot S_1^+) = \frac{k-1}{r-1}(n-k+1)+O(1).
\]
\end{cor} 

\section{Berge star forests}\label{berge-section}

A {\it system of distinct representatives} (or {\it SDR} for short) in a hypergraph $\mathcal{G}$ is a collection of $c$ distinct hyperedges $h_1,h_2,\dots, h_c$ and $c$ distinct vertices $x_1,x_2,\dots, x_c$ such that $x_i \in h_i$ for all $i$. We call $x_i$ the {\it representative} of $h_i$.
The {\it size}
of an SDR is the number of hyperedges. We will also use the term {\it $c$-SDR} to refer to an SDR of size $c$.

Let $\mathcal{G}$ be a Berge-$G$ and let $g$ be an arbitrary bijection from $\mathcal{G}$ to $G$ such that $g(h) \subset h$, i.e, $g$ is a bijection that establishes that $\mathcal{G}$ is a Berge-$G$. We call the graph formed by the image of $g$ a {\it skeleton} of $\mathcal{G}$. Thus, a skeleton of $\mathcal{G}$ is a copy of $G$ embedded into the vertex set of $\mathcal{G}$.

 Given a hypergraph $\mathcal{G}$ and a vertex $x$, the {\it link hypergraph} $\mathcal{G}_x$ is defined as
\[
\mathcal{G}_x = \{h -x \,| \, h \in \mathcal{G}, x \in h\}.
\]
Observe that for a vertex $x$ in a hypergraph $\mathcal{G}$, the existence of an $\ell$-SDR in the link hypergraph $\mathcal{G}_x$ corresponds exactly to a Berge-$S_\ell$ with center $x$ in $\mathcal{G}$ as the $x$ and the representatives of the $\ell$-SDR form the skeleton of a Berge-$S_\ell$. 

We begin with a two lemmas establishing degree conditions for the existence of a Berge-$S_\ell$.

\begin{lemma}\label{easy-sdr-lemma}
    Fix integers $r \geq 3$ and $\ell \leq r$ and let $\mathcal{G}$ be an $r$-uniform hypergraph. If $x$ is a vertex of degree $d(x) \geq \ell$ in $\mathcal{G}$, then there exists a Berge-$S_\ell$ with center $x$.
\end{lemma}

\begin{proof}
We will show that $\mathcal{G}_x$ contains an $\ell$-SDR which implies that $\mathcal{G}$ contains a Berge-$S_\ell$ with center $x$. Suppose there is no $\ell$-SDR. Then by Hall's theorem, there exists a collection of hyperedges $h_1,\dots, h_t$ with $1 \leq t \leq \ell$ such that
\[
\left|\bigcup_{i=1}^t h_i\right| < t.
\]
Each $h_i$ contains $r-1$ vertices, so $\left|\cup h_i\right| \geq r-1 \geq 2$. This implies that $t \geq 2$. As any pair of hyperedges span at least $r$ vertices we have $\left|\cup h_i\right| \geq r \geq \ell$. This implies that $t > \ell$, a contradiction.
\end{proof}

The next lemma deals with the case when $\ell>r$.

\begin{lemma}\label{sdr-lemma}
	Fix integers $\ell > r\geq 3$ and let $\mathcal{G}$ be an $r$-uniform hypergraph.
	
	\begin{enumerate}
		\item[{\bf (1)}] If $x$ is a vertex of degree $d(x) >\binom{\ell-1}{r-1}$, then 
		there exists a Berge-$S_\ell$ with center $x$.
		
		\item[{\bf (2)}] If $x$ is a vertex of degree $d(x) = \binom{\ell-1}{r-1}$ such that the neighborhood of $x$ contains at least $\ell$ vertices, then there exists a Berge-$S_\ell$ with center $x$.
	\end{enumerate}

\end{lemma}

\begin{proof}
	Consider the $(r-1)$-uniform link hypergraph $\mathcal{G}_x$. By the degree condition on $d(x)$ in {\bf (1)} or {\bf (2)} we have $|E(\mathcal{G}_x)| \geq \binom{\ell-1}{r-1}$.
	Suppose the maximum SDR is of size $c < \ell$. Let $f_1,f_2,\dots, f_c$ be the hyperedges of a $c$-SDR and let $S$ be the representatives. Let $f$ be an arbitrary hyperedge distinct from $f_1,\dots, f_c$. Observe that $f$ is contained in $S$ as otherwise we can easily form a $(c+1)$-SDR with $f$ which contradicts the maximality of $c$. Therefore, every hyperedge not part of the $c$-SDR must be contained in $S$. By condition {\bf (1)} or {\bf (2)} we have at least one $f_i$, say $f_1$, that is not contained in $S$.
	Let $y \in f_1 \setminus S$ and $z \in S$ be the representative of $f_1$.
	If $f$ contains $z$, then we can form a $(c+1)$-SDR by changing the representative of $f_1$ to $y$ and allowing $z$ to be the representative of $f$, a contradiction. Thus, $f$ does not contain $z$.
	
	We distinguish two cases.
	
	\bigskip
	
	{\bf Case 1:}  $f_1$ is the only hyperedge of the $c$-SDR not contained in $S$.
	
	\bigskip
	
	If no hyperedge different from $f_1$ is incident to $z$, then
     the number of hyperedges is at most
	\[
	1+\binom{|S \setminus \{z\}|}{r-1} \leq 1 + \binom{\ell-2}{r-1} < \binom{\ell-1}{r-1}
	\]
	which contradicts the bound on $|E(\mathcal{G}_x)|$.
	If there is a hyperedge (other than $f_1$) incident to $z$, then it must be a hyperedge of the $c$-SDR. So suppose $f_2$ is incident to $z$. Now if $f$ is incident to the representative $z'$ of $f_2$, then we can form a larger SDR by allowing $z'$ to be the representative of $f$, $z$ to be the representative of $f_2$ and $y$ to be the representative of $f_1$, a contradiction. Therefore, $f$ is disjoint from $z$ and $z'$. Then the number of hyperedges is at most
	\[
	c + \binom{|S \setminus \{z,z'\}|}{r-1} \leq (\ell-1) + \binom{\ell-3}{r-1} < \binom{\ell-1}{r-1}
	\]
	which contradicts the bound on $|E(\mathcal{G}_x)|$.
	
	\bigskip
	
	{\bf Case 2:} There are at least two hyperedges of the $c$-SDR, say $f_1$ and $f_2$, not contained in $S$.
	
	\bigskip
	
	If $z$ and $z'$ are the representatives of $f_1$ and $f_2$, respectively, then $f$ cannot contain $z$ or $z'$ as otherwise we can form a larger SDR. Therefore, the number of hyperedges is at most
	\[
	c + \binom{|S \setminus \{z,z'\}|}{r-1} \leq (\ell-1) + \binom{\ell-3}{r-1} < \binom{\ell-1}{r-1}
	\]
	which contradicts the bound on $|E(\mathcal{G}_x)|$.
\end{proof}	
	



\begin{lemma}\label{3lem}
	Fix integers $\ell>r \geq 3$ and 
	suppose $\mathcal{G}$ is a Berge-$S_\ell$-free $r$-uniform hypergraph. 
	If $x$ is a vertex of degree $d(x)=\binom{\ell-1}{r-1}$, then the link hypergraph $\mathcal{G}_x$ is a $K_{\ell-1}^{r-1}$.
\end{lemma} 

\begin{proof}
	Suppose that the link hypergraph $\mathcal{G}_x$ is not a $K_{\ell-1}^{r-1}$, then the neighborhood of $x$ must contain more than $\ell-1$ vertices as $d(x)=\binom{\ell-1}{r-1}$.
	Now we may apply Lemma~\ref{sdr-lemma} to find a Berge-$S_\ell$ in $\mathcal{G}$, a contradiction.
\end{proof}

We begin with a simple application of Lemma~\ref{sdr-lemma} which was proved in \cite{GMP}. 

\begin{theorem}[Gerbner, Methuku and Palmer \cite{GMP}]\label{berge-sk}
	Fix integers $\ell\geq 1$ and $r\geq 2$.
	\begin{enumerate}
		\item[{\bf (1)}] If $\ell>r$, then
		\[
		\operatorname{ex}_r(n, \textup{Berge-}S_\ell) \leq \binom{\ell}{r}\frac{n}{\ell}.
		\]
		Furthermore, this bound is sharp whenever $\ell$ divides $n$.
		
		\item[{\bf (2)}] 	If $\ell\leq r$, then
		\[
		\ex_r(n,\textup{Berge-}S_\ell) \leq \frac{\ell-1}{r}n.
		\]
		Furthermore, this bound is sharp whenever $r$ divides $n$.
		
	\end{enumerate}

\end{theorem}

\begin{proof}
	For the sharpness assertion in Case {\bf (1)} consider the $r$-uniform hypergraph consisting of $\frac{n}{\ell}$ pairwise disjoint complete hypergraphs $K_\ell^r$. A Berge-$S_\ell$ necessarily contains at least $r+1$ vertices (for $\ell\geq 2$), so no such complete hypergraph contains a Berge-$S_\ell$. The number of hyperedges in this construction is exactly $\frac{n}{\ell} \binom{\ell}{r}$. Note that when $\ell$ does not divide $n$ we may still construct many disjoint copies of $K_\ell^r$ and an additional smaller complete hypergraph.
	
	For the upper bound in Case {\bf (1)} suppose that $\mathcal{H}$ is an $r$-uniform $n$-vertex hypergraph with $|E(\mathcal{H})| > \frac{n}{\ell}\binom{\ell}{r}$ hyperedges. The average degree of $\mathcal{H}$ is 
	\[
	d(\mathcal{H}) > \frac{r}{n}\binom{\ell}{r}\frac{n}{\ell} = \binom{\ell-1}{r-1}.
	\]
	This implies that $\mathcal{H}$ contains a vertex $x$ of degree greater than $\binom{\ell-1}{r-1}$.
	Applying Lemma~\ref{sdr-lemma} to the hyperedges incident to $x$ gives a Berge-$S_\ell$ in $\mathcal{H}$, a contradiction.

	The upper bound in Case {\bf (2)} follows directly Lemma~\ref{easy-sdr-lemma}. The sharpness assertion follows from the the fact that  $r$-uniform $(\ell-1)$-regular $n$-vertex hypergraphs exist when $r$ divides $n$ and $n>r$. We give a concrete example here. Arrange $n$ vertices around a circle. By considering intervals of $r$ consecutive vertices we can partition the $n$ vertices into $n/r$ classes each of size $r$. If we begin this partition from different starting vertices we may create $\ell-1<r$ different partitions such that no pair of partitions has a class in common. It is easy to see that this collection of $(\ell-1)\frac{n}{r}$ total partition classes forms an $r$-uniform $(\ell-1)$-regular hypergraph on $n$ vertices.
\end{proof}

We now begin our proof of Theorem~\ref{bergekS}. We first consider the case when the uniformity $r$ is large. 

\begin{theorem}\label{large-r}
Fix integers $\ell,k \geq 1$ and $r \geq \ell+k-1$. Then for $n$ large enough,
	\[
	\ex_r(n,\textup{Berge-}k \cdot S_{\ell}) \leq \frac{\ell-1}{r-k+1}(n-k+1).
	\]
	Furthermore, this bound is sharp whenever $r-k+1$ divides $n-k+1$.
\end{theorem}

\begin{proof}
Let us begin with the lower bound. We construct a hypergraph $\mathcal{H}^*$ as follows. Consider an $(\ell-1)$-regular $(r-k+1)$-uniform hypergraph on a vertex set $B^*$ of size $n-k+1$. Such a hypergraph exists as $r-k+1$ divides $n-k+1$ and $r \geq \ell-1$. Add a fixed set $A^*$ of $k-1$ vertices to each hyperedge to form an $r$-uniform hypergraph on $n$ vertices. Each vertex in $B^*$ has degree at most $\ell-1$, so the skeleton of any Berge-$S_\ell$ must use at least one vertex from $A^*$. Therefore, $\mathcal{H^*}$ is Berge-$k \cdot S_\ell$-free and has
\[
|E(\mathcal{H}^*)| = \frac{\ell-1}{r-k+1}(n-k+1).
\]

Now we prove the upper bound.
Let $\mathcal{H}$ be an $r$-uniform $n$-vertex hypergraph with no Berge-$k\cdot S_\ell$.
	 Let $A$ be the set of vertices in $\mathcal{H}$ of degree greater than some fixed (large enough) constant $D=D(\ell,k,r)$ and let $B = V(\mathcal{H})-A$ be the remaining vertices. Put $c=|A|$.
	 
	 \begin{claim}\label{c-claim}
	 	$c \leq k-1$.
	 \end{claim}
	 
	 \begin{proof}
	 Suppose $c \geq k$ and let $A' \subseteq A$ be a set of $k$ vertices of degree at least $D$.
	 As $D$ is large enough, it is easy to see that there is a Berge-$S_\ell$ with center in $A'$ but whose skeleton is otherwise disjoint from $A'$. 
	 Now suppose that we have found a Berge-$(k-1) \cdot S_\ell$ whose skeleton intersects $A'$ in $(k-1)$ vertices (namely, only the centers of each star in the skeleton is in $A'$). This Berge-$(k-1) \cdot S_\ell$ has $(k-1)\ell$ hyperedges and its skeleton spans $(k-1)(\ell+1)$ vertices.
	 Let $x$ be the vertex in $A'$ not in the skeleton of this Berge-$(k-1) \cdot S_\ell$. Let us remove the hyperedges of the Berge-$(k-1) \cdot S_\ell$ from $\mathcal{H}$.
	 As $D$ is large enough, Lemma~\ref{sdr-lemma} implies that among the remaining hyperedges there is Berge-$S_{(k-1)(\ell+1)+\ell}$, denoted $\mathcal{S}$, with center $x$. Thus $\mathcal{S}$ is hyperege-disjoint from the Berge-$(k-1) \cdot S_\ell$. However, the skeleton of both of these subhypergraphs may intersect.
	 At most $(k-1)(\ell+1)$ vertices of the skeleton of $\mathcal{S}$ are shared with the skeleton of the Berge-$(k-1) \cdot S_\ell$. Therefore, there is a Berge-$S_\ell$ whose skeleton is disjoint from the skeleton of the Berge-$(k-1) \cdot S_\ell$. In particular, we have a Berge-$k \cdot S_\ell$, a contradiction.
	  \end{proof}
	  
	 \begin{claim}\label{b-avg-deg}
	 If $c=k-1$, then each vertex in $B$ has degree at most $\ell-1$. If $c<k-1$, then for any $\epsilon>0$ and $n$ large enough
	     the average degree of the vertices in $B$ is at most $\ell-1+\epsilon$.
	 \end{claim}
	 
	 \begin{proof}
	    First consider the case when $c=k-1$ and suppose there is a vertex $x$ in $B$ of degree at least $\ell$. Let us remove vertices from the hyperedges incident to $x$ so that they are disjoint from $A$ and of size exactly $\ell$. This is possible as $r \geq \ell+k-1$. Now, by Lemma~\ref{easy-sdr-lemma}, there is a Berge-$S_\ell$ with center $x$ among these $\ell$-uniform hyperedges. The skeleton of this Berge-$S_\ell$ is necessarily contained in $B$. Clearly this Berge-$S_\ell$ corresponds to a Berge-$S_\ell$, denoted $\mathcal{S}$, in $\h$ whose skeleton is contained in $B$.
	   Now, as the degrees of the vertices in $A$ are large enough, we can construct a Berge-$(k-1) \cdot S_\ell$ whose hyperedges and skeleton are disjoint from those of $\mathcal{S}$. Therefore, $\h$ contains Berge-$k\cdot S_\ell$, a contradiction.
	     
	     Now consider the case when  $c<k-1$ and suppose that the average degree of the vertices in $B$ is $\ell-1+\epsilon$ for some absolute constant $0 < \epsilon \leq 1$.
	     Thus, the sum of degrees of the vertices in $B$ is $(\ell-1+\epsilon)(n-c)$.
	     Let $s$ be the number of vertices of degree at most  $\ell-1$ in $B$. Recall that the vertices in $B$ have degree at most $D$.
	Thus, the sum of degrees in $B$ is at most $(\ell-1) s + D(n-c-s)$. Combining these estimates and solving for $s$ gives
	\[
	s\leq\frac{D-\ell+1 -\epsilon}{D-\ell+1 }(n-c) = (1-\epsilon')(n-c)
	\]
	for some $\epsilon'>0$ depending only on $\epsilon, \ell, k$ and $r$.
	Therefore, the number of vertices of degree at least $\ell$ is $\epsilon' n = \Omega(n)$.
	
	Let us call a pair of vertices $x,y \in B$  {\it far} if they do not share a common neighbor in $B$ (they may still have a common neighbor in $A$). As the vertices in $B$ have constant maximum degree we can find a subset $B'$ of size $\Omega(n)$ such that all vertices have degree at least $\ell$ and all pairs of vertices are far.

    For each vertex $x$ in $B'$ there is a Berge-$S_{\ell}$ with center $x$. The hyperedges of this Berge-$S_\ell$ may intersect $A$.
    However, as $r \geq \ell+k-1 > \ell + c$ we can find a skeleton of this Berge-$S_\ell$ that does not include a vertex of $A$.
     As any two vertices in $B'$ do not share a common neighbor in $B$, we have a collection of hyperedge-disjoint copies of a Berge-$S_\ell$. As $n$ is large enough we can find a Berge-$k \cdot S_{\ell}$, a contradiction.
	 \end{proof}
	 
	 Now let us estimate the number of hyperedges in $\mathcal{H}$. First observe that 
	 \[
	 \sum_{x \in A} d(x) \leq c|E(\mathcal{H})|
	 \]
	 as each hyperedge incident to $A$ is counted at most $c=|A|$ times.
	 Let $d(B)$ be the average degree of the vertices in $B$ and observe that
	 \begin{align*}
	 r|E(\mathcal{H})| & = \sum_{x \in B} d(x) + \sum_{x \in A} d(x) \\
	 & \leq d(B)(n-c) + c |E(\mathcal{H})|.
	 \end{align*}
	 Solving for $|E(\mathcal{H})|$ gives
	 \begin{equation}\label{edge-bound}
	 |E(\mathcal{H})| \leq \frac{d(B)}{r-c}(n-c).
	 \end{equation}
	 Now, by Claim~\ref{b-avg-deg} we can choose $\epsilon$ small enough and $n$ large enough so that the bound in (\ref{edge-bound}) is maximized when $c=k-1$. Thus,
	 \[
	 |E(\mathcal{H})| \leq \frac{\ell-1}{r-k+1}(n-k+1)
	 \]
	 for $n$ large enough.
\end{proof}

We now consider the case when $r \leq \ell+k-1$.
We begin with a construction of an $r$-uniform $n$-vertex Berge-$k \cdot S_\ell$-free hypergraph. Fix integers $\ell,k\geq 1$ and $r\geq 2$ such that $r \leq \ell+k-1$.
Put $n-k+1 = q \cdot \ell + t$ for $0 \leq t < \ell$. 

Let $A^*$ be a set of $k-1$ vertices and $B^*$ be a set of $n-k+1$ vertices. Partition the vertices of $B^*$ into $q$ classes of size $\ell$ and (if $t>0$) a single class of size $t< \ell$. For each partition class $S^*$ of $B^*$ we form a complete $r$-uniform hypergraph $K_{\ell+k-1}^r$ (or $K_{t+k-1}^r$) on the vertices of $A^* \cup S^*$. 
Let $\mathcal{H}({n,\ell,k,r})$ be the resulting hypergraph.
The number of hyperedges in $\mathcal{H}({n,\ell,k,r})$ is exactly 
\[
|E({n,\ell,k,r})|=\left(\binom{\ell+k-1}{r}-\binom{k -1}{r}\right)\frac{n-k+1-t}{\ell}+\binom{t+k-1}{r}.
\]
The skeleton of any Berge-$S_\ell$ in $\mathcal{H}({n,\ell,k,r})$ must use at least one vertex of $A^*$. Therefore, there are at most $k-1$ copies of a Berge-$S_\ell$ that have vertex-disjoint skeletons, i.e., $\mathcal{H}({n,\ell,k,r})$ is Berge-$k \cdot S_\ell$-free.

\begin{theorem}\label{small-r}
Fix integers $\ell,k \geq 1$ and $r\geq 2$ such that
	$r \leq \ell+k-2$. Then for $n$ large enough, 
	\[
	\ex_r(n,\textup{Berge-}k \cdot S_{\ell}) \leq \left(\binom{\ell+k-1}{r}-\binom{k -1}{r}\right)\left \lceil \frac{n-k+1}{\ell} \right \rceil+\binom{k-1}{r}.
	\]
	Furthermore, when $\ell$ divides $n-k+1$, we have equality and (as long as $r \geq 3$) $\mathcal{H}({n,\ell,k,r})$ is the unique hypergraph achieving this bound.
\end{theorem}

\begin{proof}

In order to prove the theorem it is enough to restrict ourselves to the case when $\ell$ divides $n-k+1$.
We proceed by induction on $r$. 
The case when $r=2$ follows from Theorem~\ref{LLP-star} as a Berge-$S_\ell$ is simply a copy of the graph $S_\ell$, i.e.,
\begin{align*}
\ex_2(n, \textrm{Berge-}k\cdot S_\ell) = \ex(n,k \cdot S_\ell) &\leq\frac{\ell-1}{2}(n-k+1) + (k-1)(n-k+1)+\binom{k-1}{2}\\& = \frac{n-k+1}{\ell}\left(\binom{\ell+k-1}{2}-\binom{k-1}{2}\right)+\binom{k-1}{2}.
\end{align*}

So now let $r \geq 3$ and assume that the statement of the theorem holds for $r-1$.
Let $\ell,k,n$ satisfy the conditions of the theorem and (for simplicity of notation) put $\mathcal{H}^* = \mathcal{H}({n,\ell,k,r})$.
Let $\mathcal{H}$ be an $r$-uniform $n$-vertex hypergraph with no Berge-$k\cdot S_\ell$ and at least $|E(\mathcal{H}^*)|$ hyperedges.
	 Let $A$ be the set of vertices in $\mathcal{H}$ of degree greater than some fixed (large enough) constant $D = D(\ell,k,r)$ and let $B = V(\mathcal{H})-A$ be the remaining vertices. Put $c=|A|$. 
	 
	 \begin{claim}
	 	$c \leq k-1$.
	 \end{claim}
	 
	 \begin{proof}
	 The proof of this claim identical to that of Claim~\ref{c-claim} in the proof of Theorem~\ref{large-r}. We include the argument here to keep the proof self-contained.
	 
	 Suppose $c \geq k$ and let $A' \subseteq A$ be a set of $k$ vertices of degree at least $D$.
	 As $D$ is large enough, it is easy to see that there is a Berge-$S_\ell$ with center in $A'$ but whose skeleton is otherwise disjoint from $A'$. 
	 Now suppose that we have found a Berge-$(k-1) \cdot S_\ell$ whose skeleton intersects $A'$ in $(k-1)$ vertices (namely, only the centers of each star in the skeleton is in $A'$). This Berge-$(k-1) \cdot S_\ell$ has $(k-1)\ell$ hyperedges and its skeleton spans $(k-1)(\ell+1)$ vertices.
	 Let $x$ be the vertex in $A'$ not in the skeleton of this Berge-$(k-1) \cdot S_\ell$. Let us remove the hyperedges of the Berge-$(k-1) \cdot S_\ell$ from $\mathcal{H}$.
	 As $D$ is large enough, Lemma~\ref{sdr-lemma} implies that among the remaining hyperedges there is Berge-$S_{(k-1)(\ell+1)+\ell}$, denoted $\mathcal{S}$, with center $x$. Thus $\mathcal{S}$ is hyperege-disjoint from the Berge-$(k-1) \cdot S_\ell$. However, the skeleton of both of these subhypergraphs may intersect.
	 At most $(k-1)(\ell+1)$ vertices of the skeleton of $\mathcal{S}$ are shared with the skeleton of the Berge-$(k-1) \cdot S_\ell$. Therefore, there is a Berge-$S_\ell$ whose skeleton is disjoint from the skeleton of the Berge-$(k-1) \cdot S_\ell$. In particular, we have a Berge-$k \cdot S_\ell$, a contradiction.
	  \end{proof}

\bigskip

{\bf Case 1:} $c < k-1$.

\bigskip

    For each $x \in A$ consider the link hypergraph $\mathcal{H}_x = \{e - x \, |\, x \in e \in \mathcal{H}\}$.
	If $\mathcal{H}_x$ contains an $(r-1)$-uniform Berge-$(k-1)\cdot S_\ell$, then $\mathcal{H}$ contains a Berge-$(k-1)\cdot S_\ell$ whose skeleton does not use the vertex $x$. Now remove the hyperedges of the Berge-$(k-1)\cdot S_\ell$ from $\mathcal{H}$. As $D$ is large enough, the vertex $x$ is still incident to enough hyperedges so that we can find greedily a copy of a Berge-$S_\ell$ with center $x$ that is hyperedge-disjoint from the Berge-$(k-1)\cdot S_\ell$ and whose skeleton is disjoint from the skeleton of the Berge-$(k-1)\cdot S_\ell$. Together we have a Berge-$k\cdot S_\ell$, a contradiction. Therefore, $\mathcal{H}_x$ is an $(r-1)$-uniform Berge-$(k-1) \cdot S_\ell$-free hypergraph on $n-1$ vertices. By induction we have
	\[
	d(x)=|\mathcal{H}_x|\leq \left(\binom{\ell+k-2}{r-1}-\binom{k -2}{r-1}\right)  \frac{n-k+1}{\ell} +\binom{k-2}{r-1}.
	\]
	Let $d(B)$ be the average degree of the vertices in $B$.
	Thus,
	\begin{align*}
	&\sum_{x\in V(\mathcal{H})}{d(x)}=\sum_{x \in A}{d(x)}+\sum_{x \in B}{d(x)} \\
	&\leq  c\left( \left(\binom{\ell+k-2}{r-1}-\binom{k -2}{r-1}\right)  \frac{n-k+1}{\ell}+\binom{k-2}{r-1} \right) + (n-c)d(B).
	\end{align*}

	On the other hand, for $\mathcal{H}^*=\mathcal{H}({n,\ell,k,r})$ we have 
	\begin{align*}
	&\sum_{x\in V(\mathcal{H^*})}{d(x)}=\sum_{x \in A^*}{d(x)}+\sum_{x \in B^*}{d(x)}\\& \geq  (k-1)\left( \left(\binom{\ell+k-2}{r-1}-\binom{k -2}{r-1}\right)\frac{n-k+1}{\ell}+\binom{k-2}{r-1} \right) + (n-k+1)\binom{\ell+k-2}{r-1}.
	\end{align*}
	
	As $c< k-1$ and  $|E(\mathcal{H})| \geq |E(\mathcal{H}^*)|$, if we
	compare the coefficients of $n$ in the two inequalities, it is clear that we must have $d(B) \geq \binom{\ell+k-2}{r-1}+1$ for $n$ large enough.

	Now let us estimate the number of vertices in $B$ of degree greater than $\binom{\ell+k-2}{r-1}$.
    Let $s$ be the number of vertices of degree at most $\binom{\ell+k-2}{r-1}$ in $B$. Recall that the vertices in $B$ have degree at most $D$.
	Thus, the sum of degrees in $B$ is at most $\binom{\ell+k-2}{r-1} s + D(n-c-s)$. Combining these estimates and solving for $s$ gives
	\[
	s\leq\frac{D-\binom{\ell+k-2}{r-1}-1}{D-\binom{\ell+k-2}{r-1}}(n-c) = (1-\epsilon')(n-c)
	\]
	for some $\epsilon'>0$ not depending on $n$.
	Therefore, the number of vertices of degree greater than $\binom{\ell+k-2}{r-1}$ is $\epsilon' n = \Omega(n)$.
	
		Let us call a pair of vertices $u,v \in B$ {\it far} if they do not share a common neighbor in $B$ (they may still have a common neighbor in $A$). As the vertices in $B$ have constant maximum degree we can find a subset $B'$ of size $\Omega(n)$ such that all vertices have degree greater than $\binom{\ell+k-2}{r-1}$ and all pairs of vertices are far.
		
    By Lemma~\ref{sdr-lemma},
    for each vertex $u$ in $B'$ there is a Berge-$S_{\ell+k-1}$ with center $u$. Therefore, there is a Berge-$S_\ell$ with center $u$ whose skeleton is disjoint from $A$. 
     As any two vertices in $B'$ do not share a common neighbor in $B$ we have a collection of hyperedge-disjoint copies of Berge-$S_\ell$. Thus, as $n$ is large enough we can find a Berge-$k \cdot S_{\ell}$, a contradiction.

\bigskip

{\bf Case 2:} $c=k-1$.

\bigskip

In this case we have that each vertex in $B$ has degree at most $\binom{\ell+k-2}{r-1}$. Indeed, if there is a vertex $x$ of degree greater than $\binom{\ell+k-2}{r-1}$, then there is a Berge-$S_{\ell+k-1}$ with center $x$ by Lemma~\ref{sdr-lemma}. The skeleton of this Berge-$S_{\ell+k-1}$ uses at most $k-1$ vertices from $A$, so there remains a Berge-$S_\ell$, denoted $\mathcal{S}$ whose skeleton is contained in $B$. 
As $D$ is large enough, we may construct $k-1$ more pairwise hyperedge-disjoint copies of a Berge-$S_\ell$ (with pairwise vertex-disjoint skeletons) that are hyperedge-disjoint from $\mathcal{S}$ and whose skeletons are vertex-disjoint from the skeleton of $\mathcal{S}$. Therefore, we have a Berge-$k \cdot S_\ell$ in $\mathcal{H}$, a contradiction.

We distinguish two subcases based on the degrees in $B$. 

\bigskip

{\bf Case 2.1:} There exists a vertex $x \in B$ with $d(x)<\binom{\ell+k-2}{r-1}$.

\bigskip

Let us compare $\mathcal{H}$ to the construction $\mathcal{H}^*$. Every vertex in $B$ has degree at most $\binom{\ell+k-2}{r-1}$ while every vertex in $B^*$ has degree exactly $\binom{\ell+k-2}{r-1}$. 

As $|E(\mathcal{H})| \geq |E(\mathcal{H}^*)|$ this implies that there exists a vertex $y \in A$ and a vertex $y^* \in A^*$ such that $d(y) > d(y^*)$. Define two $j$-uniform multi-hypergraphs as follows:
\[
E_j^y = \left\{e \setminus A\, \mid \, y \in e \text{ and } |e \setminus A| = j \right\}
\text{ and }
E_j^{y^*} = \left\{e \setminus A^*\, \mid \, y \in e \text{ and } |e \setminus A^*| = j \right\}.
\]
Note that the hyperedges in $E^y_j$ have multiplicity at most $\binom{k-2}{r-1-j}$ and those in $E^{y^*}_j$
have multiplicity exactly $\binom{k-2}{r-1-j}$.

Observe that when $k-1 \geq r-1$, each vertex of $B^*$ is in a hyperedge with each subset of $A^*$ of size $r-1$. This implies that $|E^{y^*}_{1}|\geq |E^{y}_{1}|$. When $k-1 < r-1$, then $E^{y^*}_{1}= E^{y}_{1} = \emptyset$.
Therefore, as $d(y) > d(y^*)$, we have that $|E^y_j|>|E^{y^*}_j|$ for some $j\geq 2$. Now let $E$ be the $j$-uniform hypergraph resulting from deleting all repeated hyperedges in $E^y_j$.

Thus, the number of hyperedges in the $j$-uniform hypergraph $E$ is
\begin{align*}
|E| &\geq \binom{k-2}{r-1-j}^{-1}|E_j^{y}| \\ & > \binom{k-2}{r-1-j}^{-1}|E^{y^*}_j|  
= \binom{k-2}{r-1-j}^{-1}\binom{k-2}{r-1-j} \frac{n-k+1}{\ell} \binom{\ell}{j} \\
 & = \frac{n-k+1}{\ell} \binom{\ell}{j}  \geq \ex_j(n-k+1,\textrm{Berge-}S_\ell).
\end{align*}

Therefore, there is a $j$-uniform Berge-$S_\ell$ in $E$ on the vertices of $B$. As each hyperedge of $E$ is contained in a hyperedge of $\mathcal{H}$, this corresponds to a Berge-$S_\ell$ in $\mathcal{H}$ whose skeleton is contained in $B$.
As before, the degree condition on the vertices in $A$ guarantees the existence of a Berge-$(k-1) \cdot S_\ell$ that together with this Berge-$S_\ell$ forms a Berge-$k \cdot S_\ell$ in $\mathcal{H}$, a contradiction.

\bigskip

{\bf Case 2.2:} For every vertex $x \in B$ we have $d(x)=\binom{\ell+k-2}{r-1}$.

\bigskip

If the neighborhood $N(x)$ of $x$ contains more than $\ell+k-2$ vertices, then by Lemma~\ref{sdr-lemma} there is a Berge-$S_{\ell+k-1}$ with center $x$. As before, we can use this Berge-$S_{\ell+k-1}$ to show the existence of a Berge-$k\cdot S_{\ell}$, a contradiction.
Therefore, $|N(x)| = \ell+k-2$ and Lemma~\ref{3lem} implies that the link hypergraph $\mathcal{H}_x$ is a $K_{\ell+k-2}^{r-1}$. The vertices of this $K_{\ell+k-2}^{r-1}$ are exactly $N(x)$. If $N(x)$ intersects $A$ in fewer than $k-1$ vertices, then we have a Berge-$S_\ell$ whose skeleton is contained in $B$ which we can combine with the Berge-$(k-1)\cdot S_\ell$ with centers in $A$ to form a Berge-$k \cdot S_\ell$, a contradiction. 
Therefore, $N(x) \cap A = A$ and $|N(x) \cap B| = \ell-1$. Now let $y$ be a vertex in $N(x) \cap B$. For each $z$ different from $y$ in $N(x) \cup \{x\}$, there is a hyperedge containing $x$ and $y$. Therefore, the neighborhood of $y$ contains $N(x) \cup \{x\}$. If $N(y)$ is any larger, then by Lemma~\ref{sdr-lemma} we can find a Berge-$S_{\ell+k-1}$ with center $y$, which again leads to a contradiction.

This implies that $x$ is contained in a complete graph $K^r_{\ell+k-1}$ that intersects $A$ in exactly $k-1$ vertices. This holds for every vertex $x \in B$, so $\mathcal{H}$ has the exact structure as the construction $\mathcal{H}^*$.
\end{proof}

With some additional argument it seems likely that one can show that $\h({n,\ell,k,r})$ is the unique extremal hypergraph for Berge-$k \cdot S_\ell$ even when we do not have the appropriate divisibility condition on $n+k+1$.

We conclude this section with an easy consequence of Theorem~\ref{small-r}. We need a further definition. For fixed graphs $H$ and $F$ let $\ex(n,H,F)$ denote the maximum number of copies of a subgraph $H$ in an $n$-vertex $F$-free graph. This generalization of the classical Tur\'an extremal number $\ex(n,F)$ was introduced by Alon and Shikhelman \cite{AS}. These {\it generalized Tur\'an numbers} are closely related to Tur\'an numbers for Berge hypergraphs. See \cite{G_P2, Palmer, GMP} for details.

If $G$ is an $n$-vertex $F$-free graph, then we can define an $r$-uniform hypergraph $\h$ with the same vertex set as $G$ and an $r$-set in $\h$ is a hyperedge if and only if it is the vertex set of a $K_r$ in $G$. It is easy to see that as $G$ is $F$-free, then $\h$ contains no Berge-$F$. Thus
\begin{equation}\label{counting-bound}
 \ex(n,K_r,F) \leq \ex_r(n,\textrm{Berge-}F)
\end{equation}
 holds in general. Therefore, the upper bound in Theorem~\ref{small-r} gives the following corollary. The construction is nearly identical to the hypergraph $\h({n,\ell,k,r})$. We leave the details to the reader.
 
 \begin{cor}\label{kr}
    Fix integers $\ell,k \geq 1$, $r\geq 3$. 
	If $r \leq \ell+k-2$, $\ell$ divides $n-k+1$ and $n$ is large enough then,
	\[
	\ex(n,K_r,k \cdot S_{\ell}) = \left(\binom{\ell+k-1}{r}-\binom{k -1}{r}\right) \frac{n-k+1}{\ell} +\binom{k-1}{r}.
	\]
\end{cor}

When $r \geq \ell+k-1$, Theorem~\ref{large-r} gives an upper bound on $\ex(n,K_r,k \cdot S_\ell)$ that is linear in $n$. However, it seems likely that the correct bound is a constant depending on $\ell,k$ and $r$. Indeed, we conjecture that the unique $k \cdot S_\ell$-free graph maximizing the number of copies of $K_r$ is a complete graph on $k(\ell+1)-1$ vertices (to be precise here we should assume every edge is in a $K_r$). This conjecture is supported by a theorem of Wang \cite{Wa} that implies
\[
\ex(n,K_r,k\cdot S_1) = \binom{2k-1}{r}
\]
when $r \geq k+2$ and $n \geq 2k-1$ which is essentially the case of the conjecture when $\ell=1$. The shifting method employed in \cite{Wa} may be effective in proving this conjecture. For more on the problem of forbidding disjoint copies of a graph see Gerbner, Methuku and Vizer \cite{GMV2}.

\section*{Acknowledgments}
The authors would like to thank D\'aniel Gerbner and Abhishek Methuku for several useful discussions.

\end{document}